\documentclass[]{amsart}

\usepackage{etex}

\usepackage[USenglish]{babel}
\usepackage[utf8]{inputenc}

\usepackage{amsfonts}
\usepackage{amssymb}
\usepackage{amsthm}
\usepackage{amscd}
\usepackage{amsxtra}
\usepackage[all]{xy}
\usepackage{todonotes}
\usepackage{comment}
\usepackage{caption}
\usepackage{tikz}

\usepackage{ upgreek }
\usepackage{graphicx}
\usepackage{float}

\usepackage{enumerate}
\usepackage{dsfont}
\usepackage{epic}
\usepackage[dvips]{rotating}

\usepackage{todonotes}
\usepackage[pdftitle={},
	    pdfauthor={Matteo Costantini, Andre Kappes},
	    linktocpage=true
	    ]{hyperref}

\usetikzlibrary{intersections,calc,arrows,backgrounds,decorations.markings}
\usepackage{pgfplots}

\DeclareMathOperator{\Moduli}{\mathcal{M}}       
\DeclareMathOperator{\Teich}{\mathcal{T}}        
\DeclareMathOperator{\Stratum}{\Omega \mathcal{M}}      
\DeclareMathOperator{\XFam}{\mathcal{X}}      
\DeclareMathOperator{\YFam}{\mathcal{Y}}      
\DeclareMathOperator{\HFam}{\mathcal{H}}      
\DeclareMathOperator{\HH}{\mathbb{H}}            
\DeclareMathOperator{\PP}{\mathbb{P}}            

\newcommand{\dual}{\vee}		         

\DeclareMathOperator{\OX}{\mathcal{O}}           

\DeclareMathOperator{\KBundle}{\mathcal{K}}       
\DeclareMathOperator{\LBundle}{\mathcal{L}}       
\DeclareMathOperator{\VBundle}{\mathcal{V}}       

\DeclareMathOperator{\K}{\mathbb{K}}		  
\DeclareMathOperator{\CC}{\mathbb{C}}		  
\DeclareMathOperator{\RR}{\mathbb{R}}		  
\DeclareMathOperator{\ZZ}{\mathbb{Z}}		  
\DeclareMathOperator{\QQ}{\mathbb{Q}}		  

\DeclareMathOperator{\Sym}{\mathrm{Sym}}         
\DeclareMathOperator{\End}{\mathrm{End}}         
\DeclareMathOperator{\PGL}{\mathrm{PGL}}         
\DeclareMathOperator{\Jac}{\mathrm{Jac}}         
\DeclareMathOperator{\Pic}{\mathrm{Pic}}         
\DeclareMathOperator{\diag}{\mathrm{diag}}         

\newcommand{\dd}{\mathrm{d}}			 	
\newcommand{\pder}[2][]{\frac{\partial #1}{\partial #2}}  

\renewcommand{\div}{\mathrm{div}}            
\DeclareMathOperator{\isom}{\cong}               

\DeclareMathOperator{\tensor}{\otimes}           

\DeclareMathOperator{\ord}{\mathrm{ord}}         
\DeclareMathOperator{\Res}{\mathrm{Res}}         
 
\DeclareMathOperator{\Tor}{\mathrm{Tor}}
\DeclareMathOperator{\Prim}{\mathrm{P}}

\newcommand{\ol}[1]{\overline{#1}}	  	 

\newcommand{\eg}{e.\,g.\ }		
\newcommand{\ie}{i.\,e.\ }		

\DeclareMathOperator{\Hcoh}{\mathrm{H}}   

\excludecomment{draftcomment}

\theoremstyle{plain}
\newtheorem{prop}{Proposition}[section]
\newtheorem{thm}[prop]{Theorem}
\newtheorem{lem}[prop]{Lemma}
\newtheorem{cor}[prop]{Corollary}

\theoremstyle{definition}

\theoremstyle{remark}
\newtheorem{rem}[prop]{Remark}

\numberwithin{equation}{section}

\title[Equation of the $(2,3,4)$-Teichm\"uller curve]{The Equation of the Kenyon-Smillie $(2,3,4)$-Teichm\"uller curve}

\author[M. Costantini, A. Kappes]{Matteo Costantini and Andr\'e Kappes}
\address{Institut f\"ur Mathematik, Robert-Mayer-Str. 6--8,
 Goethe-Universit\"at Frankfurt/Main}
\email{costanti@math.uni-frankfurt.de}
\email{andre.kappes@gmx.net}
\date{\today}

\thanks{}
\begin{document}

\begin{abstract}
We compute the algebraic equation of the universal family over 
the Kenyon-Smillie $(2,3,4)$-Teichm\"uller curve and we prove that the equation is correct in two different ways.

Firstly, we prove it in a constructive way via linear conditions imposed by three special points of the Teichmüller curve. Secondly, we verify that the equation is correct by computing its associated Picard-Fuchs equation.  

We also notice that each point of the Teichmüller curve has a hyperflex and we see that the torsion map is a central projection from this point.
\end{abstract}

\maketitle

\section{Introduction}

Almost all known primitive Teichm\"uller curves fall in very few series. Currently, an infinite series in genus $2$ is known by independent work of Calta and McMullen (\cite{mcmullenbild}, \cite{calta}), which generalizes to the construction of the infinite Prym families (\cite{mcmullenprym}). In addition, there is the infinite series of Bouw-M\"oller curves (\cite{bouwmoel}) generalizing \cite{veech89} and \cite{ward}. Moreover, recently McMullen, Mukamel and Wright \cite{mcmullenmukamelwright16} discovered a new series of Teichm\"uller curves in genus $4$.

There are only two known primitive Teichm\"uller curves that do not belong to any of these families. One of them parametrizes all affine deformations $\{(S_t,\omega_t)\}_t$ of the translation surface $(S,\omega)\in \Stratum_3(3,1)$ that is obtained from unfolding a Euclidean triangle  with angles $(\tfrac{2\pi}{9}, \tfrac{3\pi}{9}, \tfrac{4\pi}{9})$.  It was discovered by Kenyon and Smillie \cite{kenyon-smillie}, who proved that $(S,\omega)$ is a lattice surface with  Veech group equal to the triangle group $\Delta(9,\infty,\infty)$. 
The translation surface $(S,\omega)$ is the order $9$ orbifold point of its associated Teichmüller curve, which is uniformized by $\HH/\Delta(9,\infty,\infty)$.

\subsection*{The equation.}
In this paper, we discuss the Kenyon-Smillie $(2,3,4)$-Teichmüller curve from an algebro-geometric perspective.
First, we derive the equations of the algebraic curves $\{S_t\}_{t}$ parametrized by this Teichm\"uller curve.

\begin{thm}
\label{thm:mainthm}
The universal family over the complement of the orbifold point of the  Kenyon-Smillie $(2,3,4)$-Teichm\"uller curve is given by the family of plane quartics  satisfying the equation  
\begin{equation}
\label{eq:maineq}
\begin{split}
X^4+ t (X^4- 3X^3Y+ 6X^3Z -3X^2Y^2- 6X^2 Y Z +  6 X^2 Z^2   +4 X Y^3  \\
 \quad -6 X Y^2 Z - 6 X Y Z^2+ X Z^3+ 3 Y^4 +3Y^3 Z)=0
\end{split}
\end{equation}
where $t$ varies in $\PP^1-\{0,1,\infty\}$.

The triple zero  of the differential $\omega_t$ is the point $P_t=(0:0:1)\in S_t$ and the simple zero  is the point $Q_t=(0:1:-1)\in S_t$. 
\end{thm}

Since this Teichmüller curve has one orbifold point, which is $t=0$ in our presentation, the family we give is universal only over the complement of this point and there is no universal family over the whole Teichmüller curve. See Equation~\eqref{eq:spar} for the universal family over a cover of the full Teichmüller curve.

We prove Theorem \ref{thm:mainthm} in two different ways. Firstly we  prove it constructing the equation using considerations on the conditions imposed by the three special points of the Teichmüller curve, the orbifold point $S_0$ and the two cusps $S_1$ and $S_{\infty}$. Secondly, we  show again that Equation  \eqref{eq:maineq} defines a Teichmüller curve by constructing the Picard-Fuchs differential equation associated with the variation of Hodge structures over the Teichmüller curve.

\subsection*{The Torsion map.}

By \cite{moellerper}, the difference between two zeros of $\omega_t$ is a torsion point of the Jacobian of $S_t$. Thus, there is a minimal $n$ such that for all $t$, the divisor $n(P_t - Q_t)$ is the divisor of a meromorphic function, which we call the \textit{torsion map} $\Tor$. Using Theorem~\ref{thm:mainthm}, we can explicitly determine this map for the Kenyon-Smillie Teichm\"uller curve and we notice that $Q_t$ is a hyperflex for every $t$. 

\begin{prop}
	\label{prop:torsion}
	The torsion map is given by the projection from $Q_t$, which is the map induced by the linear system $|K_{S_t}(-Q_t)|$. 
	The point $Q_t$ is a hyperflex for every $t$, which means that the tangent to $S_t$ in $Q_t$ is of order $4$. 
\end{prop}

\medskip
\begin{minipage}{\textwidth}
	\centering
	\includegraphics[scale=0.55]{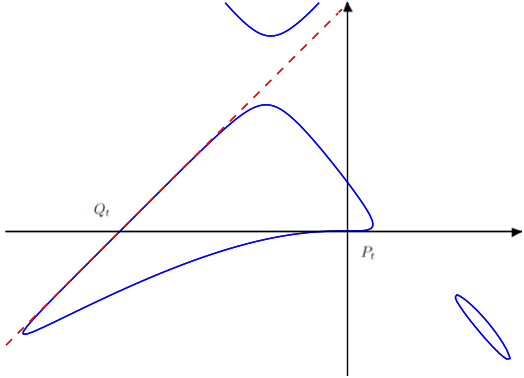}\par
	\captionof{figure}{Real points of the curve $S_t$ for $t=3$ near  $P_t=(0,0)$ and the hyperflex $Q_t=(0,-1)$.}
\end{minipage}
\medskip

In the above coordinates, the torsion map is the degree $3$ map totally ramified over $P_t$ and $Q_t$ given by
\[\Tor:S_t\to \PP^1,\quad \begin{cases}
(X:Y:Z)\mapsto (X:Y+Z)\quad \text{ for }\ (X:Y:Z)\not = (0:1:-1)\\
Q_t=(0:1:-1)\mapsto (-1:1) 
\end{cases}\]

We will give also a proof of the above Proposition in which we do not use the explicit equation of the universal family. Indeed, it is enough to know that the Teichmüller curve parametrizes quartics in $\PP^2$ together with a differential in the stratum $(3,1)$.

Notice that, since every $S_t$ has a hyperflex in $Q_t$, the Teichm\"uller curve is also in the image of the projection of the stratum $\Stratum_3(4)^{\text{odd}}$, which is a divisor in $\mathcal{M}_3$.

The initial motivation of the paper was to investigate a question of Alex Wright about the relation between real multiplication and the torsion map. One can  see such a relation in the Veech-Ward-Bouw-M\"oller Teichm\"uller curves, where real multiplication is induced by the correspondence given by the graph of an automorphism coming from the Galois normalization map (see Section~\ref{sec:torsion} for details).
As in the Veech-Ward-Bouw-M\"oller case, the Kenyon-Smillie Teichm\"uller curve is also uniformized by a triangle group. However, we will deduce from Proposition~\ref{prop:torsion} that for the  Kenyon-Smillie-Teichm\"uller curve, real multiplication does not come from the normalization of the torsion map. This gives some evidence in support of the Kenyon-Smillie example being sporadic.

\subsection*{Picard-Fuchs equation.}
Given only  Equation~\eqref{eq:maineq}, we can prove independently that it is indeed the equation of the universal family over a Teich\-m\"uller curve by computing the Picard-Fuchs equation satisfied by the periods of $\omega_t$ and by employing Möller's characterization of Teichmüller curves.

\begin{prop}
\label{prop:P-F}
The periods of $\omega_t$ are solutions of the following differential equation: 
\[\frac{16}{81t(t-1)}y+\frac{17t - 8}{9t(t - 1)}y'+y''=0.\]
\end{prop}

Note that this equation is a hypergeometric differential equation, meaning that it has three regular singular points.

The main ingredient to compute the above Picard-Fuchs equation is the well-known Griffiths-Dwork method, which is an algorithm for computing Picard-Fuchs equations of family of projective hypersurfaces.

As an immediate consequence, we obtain the following corollary.


\begin{cor}
	\label{cor:maxhiggs}
	The section $\omega_t$  defines a maximal Higgs, irreducible, rank 2 subbundle of the flat bundle defined by the variation of Hodge structures.
\end{cor}

By the characterization of Teichm\"uller curves of \cite[Theorem~5.3]{moeller06}, Corollary \ref{cor:maxhiggs} shows again that Equation~\eqref{eq:maineq} defines a Teichmüller curve.

\subsection*{Structure of the paper.}
In Section~\ref{sec:prelims}, we set up the notation and we gather known facts about the Kenynon-Smillie Teichmüller curve and the splitting of $H^0(S_t,\Omega^1)$ into one-dimensional eigenspaces given by real multiplication on the Jacobian of $S_t$.  In Section~\ref{sec:specialpoints}, the three special points, the two cusps and the orbifold point, of the Kenyon-Smillie-curve are discussed. Section~\ref{sec:equation} contains the proof of Theorem~\ref{thm:mainthm} and Proposition~\ref{prop:torsion}. The starting point  is that every non-hyperelliptic curve in genus $3$ is a canonically embedded quartic in $\PP^2$.  The main ingredient in the proof of Theorem~\ref{thm:mainthm}  is the use of a structural result for the relative canonical ring of a family of quartics of \cite{CatanesePignatelli2006}. This, together with the computation of the degree of the eigenspace bundles, enables us to obtain an a priori bound on the 
degree of the coefficients of the family. It then suffices to consider the three special points on the Teichm\"uller curve to compute these coefficients. In Section~\ref{sec:picard-fuchs}, we derive the Picard-Fuchs equation proving Proposition~\ref{prop:P-F} and we show again that Equation~\eqref{eq:maineq} defines a Teichmüller curve.

\subsection*{Notes and references.}
The strategy that we use to determine the equation of the universal family has already been successfully implemented before by 
Bouw and M\"oller \cite{BouwM10}, who have worked out the equations of two of the Weierstra\ss\ curves in genus two.
The main challenge is to find a suitable structural result about the canonical ring of quartic hypersurfaces analogous to the one used by Bouw and Möller for genus two curves.

For an algebraic description of the other Weierstra\ss\ curves in genus two with fundamental discriminant $D<100$, see \cite{mukamel}. The methods that they use can be applied to other Weierstra\ss\ curves in genus two and possibly also to the Prym curves in genus three and four.

The second Teichm\"uller curve that is not known to belong to an infinite series has been found by Vorobets \cite{voroveechalt} and is generated by the surface in $\Stratum_4(6)$ obtained from unfolding a billiard in a triangle with angles $(\tfrac{\pi}{5}, \tfrac{\pi}{3},\tfrac{7\pi}{15})$. To the best of our knowledge, no algebro-geometric construction for its Teichm\"uller curve is known so far.

It is an open question, whether these two Teichm\"uller curves are truly sporadic. 
One argument in support of this hypothesis is given in \cite[Section 7]{leininger}, 
from which one can deduce that these Teichm\"uller curves do not fit into a family, 
where the Veech group is generated by two Dehn multi-twists. 
Since not all parabolics are multi-twists, this is however not enough 
to prove sporadicity among Veech groups generated by two parabolics. 
It is also worth noting that both the Kenyon-Smillie Teichm\"uller curve 
and the Vorobets Teichm\"uller curve correspond to two of the three exceptional 
billiard triangles that are naturally attached to the sporadic Coxeter groups 
$E_6$, $E_7$ and $E_8$.

Notice that, curiously, the Kenyon-Smillie Teichmüller curve is the only known primitive Teichm\"uller curve where the zeros of the differential have different orders and the only known primitive Teichm\"uller curve that does not possess an involution negating $\omega$.

It would be interesting to find the correspondence that gives real multiplication for the Kenyon-Smillie Teichm\"uller curve. A related open question is whether real multiplication for this curve is of Hecke type as in the Veech-Ward-Bouw-M\"oller case (cf. \cite{wright13}).

\subsection*{Acknowledgements.}
We thank Martin M\"oller for suggesting this project to us and for many fruitful and valuable discussions, 
and Alex Wright for his helpful comments on an early draft of this paper and in particular for pointing out to us how peculiar the Kenyon-Smillie-Teichm\"uller curve really is. We would also like to thank an anonymous referee for the helpful comments about the presentation.
Many computations were carried out with the help of \cite{PARI}, and we thank the developers for their work.

\section{Preliminaries}\label{sec:prelims}

In this section, we gather building blocks for the proof of the main theorem. We assume that the reader is familiar with the basic notions on Teichm\"uller curves. For background reading one may consult for example  \cite{moelPCMI}.

\subsection{Orbifold uniformization and universal family}
Let $C$ denote the Kenyon-Smillie-Teichm\"uller curve associated with the flat surface $(S,\omega)$ shown in Figure~\ref{fig:KS-unfolding}.

\begin{figure}[ht]
 \begin{center}
  \begin{tikzpicture}[scale=1.4, 
	ar/.style ={decoration={             
            markings, 
            mark=at position 0.5 with {\arrow{triangle 45}}
        }, postaction={decorate}}] 
\clip (-0.5,-0.2) rectangle (2.5,5.5);

\coordinate (p0) at (0,0);
\coordinate (p1) at (20:1);
\coordinate (p2) at ($(p1)+(40:1)$);
\coordinate (p3) at ($(p2)+(60:1)$);
\coordinate (p4) at ($(p3)+(80:1)$);
\coordinate (p5) at ($(p4)+(100:1)$);
\coordinate (p6) at ($(p5)+(120:1)$);
\coordinate (p7) at (100:1);
\coordinate (p8) at ($(p7)+(80:1)$);
\coordinate (p9) at ($(p8)+(340:1)$);
\coordinate (p10) at ($(p9)+(120:1)$);
\coordinate (p11) at ($(p10)+(20:1)$);
\coordinate (p12) at ($(p11)+(160:1)$);
\coordinate (p13) at ($(p12)+(60:1)$);
\coordinate (p14) at ($(p13)+(40:1)$);
\coordinate (cdelta) at ($(p4)+(160:0.7)$);
\draw (p1)--(p8);
\draw (p7)--(p2);
\draw (p0)--(p9);
\draw (p2)--(p9);

\draw (p3)--(p10);
\draw (p4)--(p11);
\draw (p2)--(p11);
\draw (p4)--(p9);

\draw (p4)--(p13);
\draw (p5)--(p12);
\draw (p6)--(p11);

\draw[ar, name path=side51] (p1)-- node[below] {$5$} (p0);
\draw[ar, name path=side11] (p1)-- node[below] {$1$} (p2);
\draw[ar, name path=side61] (p3)-- node[below,right] {$6$} (p2);
\draw[ar, name path=side21] (p3)-- node[right] {$2$} (p4);
\draw[ar, name path=side71] (p5)-- node[right] {$7$} (p4);
\draw[ar, name path=side31] (p5)-- node[right] {$3$} (p6);
\draw[ar, name path=side72] (p7)-- node[left] {$7$} (p0);
\draw[ar, name path=side22] (p7)-- node[left] {$2$} (p8);
\draw[ar, name path=side41] (p9)-- node[above] {$4$} (p8);
\draw[ar, name path=side32] (p9)-- node[left] {$3$} (p10);
\draw[ar, name path=side52] (p11)-- node[above] {$5$} (p10);
\draw[ar, name path=side42] (p11)-- node[left] {$4$} (p12);
\draw[ar, name path=side62] (p13)-- node[left] {$6$} (p12);
\draw[ar, name path=side12] (p13)-- node[left, above] {$1$} (p14);

\draw[fill=black] (p0) circle (1pt);
\draw[fill=black] (p2) circle (1pt);
\draw[fill=black] (p4) circle (1pt);
\draw[fill=black] (p6) circle (1pt);
\draw[fill=black] (p8) circle (1pt);
\draw[fill=black] (p10) circle (1pt);
\draw[fill=black] (p12) circle (1pt);

\draw[fill=white] (p1) circle (1pt);
\draw[fill=white] (p3) circle (1pt);
\draw[fill=white] (p5) circle (1pt);
\draw[fill=white] (p7) circle (1pt);
\draw[fill=white] (p9) circle (1pt);
\draw[fill=white] (p11) circle (1pt);
\draw[fill=white] (p13) circle (1pt);

\node at (cdelta) {$\Delta$};

  \end{tikzpicture}
 \end{center}
\caption{The Kenyon-Smillie $(2,3,4)$-lattice surface resulting from unfolding the triangle $\Delta$. Sides are labeled by powers of $\zeta_9 = \exp(2\pi i/9)$ and sides with the same label are identified. The triple (simple) zero is marked by a white (black) dot.
\label{fig:KS-unfolding}
}
\end{figure}
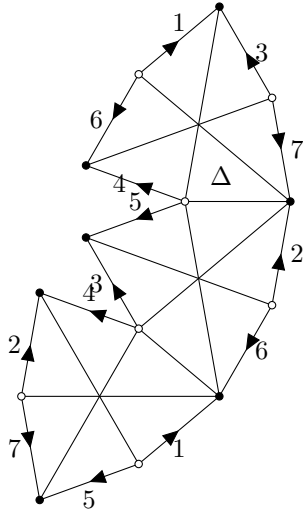

In \cite{kenyon-smillie}, Kenyon and Smillie prove that the Veech group of $(S,\omega)$ is the triangle group $\Delta(9,\infty,\infty)$. 
Thus, $C$ is uniformized by $\HH/\Delta(9,\infty,\infty)$ and the completion   of $C$ is isomorphic to $\PP^1$.
Let $t\in\PP^1$ be a parameter such that $t=0$ is the orbifold point of order $9$,  $t=1$ is the cusp correspondig to the limit deformation of the horizontal cylinder decomposition of $(S,\omega)$ and $t=\infty$ the cusp corresponding to the vertical one. 
Let $S_t$ denote the (stable) Riemann surface parametrized by $t$ and let $\omega_t$ denote the (stable) holomorphic $1$-form on $S_t$ that is obtained by affinely deforming $\omega$. 
By our choice of $t$, we have $(S,\omega) = (S_0,\omega_0)$. Let further $P_t$, $Q_t \in S_t$ be such that
\[\div(\omega_t) = 3P_t + Q_t.\]
We will abbreviate $P=P_0$ and $Q= Q_0$.

To avoid working with  orbifold line bundles, we pass to a finite \'etale covering $\tilde C$ of $C$.
In order to unfold the orbifold structure of $C$, we perform the 9-sheeted covering 
\[\tilde C \to \PP^1,\quad s\mapsto t = s^9\]
  which is totally ramified over the orbifold point and the cusp at $\infty$. The corresponding subgroup $\Gamma'$ of the Veech group is free and the quotient by $\Gamma'$ of the universal family over the Teichm\"uller disk gives the family of curves $\tilde{\phi}:\tilde{\XFam}\to \tilde{C}$. 
  By the Riemann-Hurwitz theorem the genus of $\tilde{C}$ is still $0$ and by construction it has $10$ cusps, which are given by the points $\{\zeta_9^i,\infty\}_{i=1,\dots,9}$ where $\zeta_9:=\exp(2\pi i/9)$.

Let $\ol{C}\isom \PP^1$ denote the completion of $\tilde C$.
The family $\tilde{\XFam}$ is extended to the family $\ol{\phi}:\ol{\XFam}\to \ol{C}$, by adding stable curves at the cusps of $\tilde C$.

\subsection{Absolute cohomology bundle}
The \emph{absolute cohomology bundle}, which will come up in Section~\ref{sec:picard-fuchs}, is the flat vector bundle on $\tilde C$ with fiber $H^1(\tilde\XFam_s,\CC)$ over a point $s\in\tilde C$. Formally, it is the vector bundle associated with the locally constant sheaf $R^1\tilde{\phi}_*\QQ$. It has a canonical extension, due to Deligne, to a vector bundle on all of $\ol{C}$, which we will still refer to as the absolute cohomology bundle. In the context of Teichmüller dynamics, this bundle is often called Hodge bundle, but we will not use this name, since in algebraic geometry, the Hodge bundle is the push forward of the relative canonical bundle.

\subsection{Real multiplication}
By \cite[Theorem 2.7]{moeller06}, the Jacobian of $\ol{\XFam}_s$ has real multiplication by the trace field $\K(S,\omega)$, which in our case is the cubic number field
\[\K(S,\omega) = \QQ(v),\qquad  v = 2\cos\bigl(\tfrac{2\pi}{9}\bigr).\]
In particular, $C$ is an algebraically primitive Teichm\"uller curve, \ie $g(\XFam_s) = \deg(\K/\QQ)$.
Note that $v$ is a root of $P(v) = v^3 - 3v + 1$. The field $\K(S,\omega)$ has three embeddings $\sigma_i$, $i=1,2,3$, into $\RR$. We denote the image of an element $\lambda\in \K(X,\omega)$ under the $i$-th embedding by $\lambda^{(i)}$. We order the three embeddings by requiring that
\[v^{(1)} = v,\qquad v^{(2)} = 2-v-v^2,\qquad v^{(3)} = -2 + v^2.\]

Having real multiplication implies that  the direct image  $\VBundle:=\ol{\phi}_*\omega_{\ol{\XFam}/\ol{C}}$  of the relative dualizing sheaf $\omega_{\ol{\XFam}/\ol{C}}$, is a rank $3$ vector bundle which splits as a direct sum of three line bundles
\[\VBundle = \bigoplus_{i=1}^3 \LBundle_i\]
where $\LBundle_i$ is the eigenspace bundle for real multiplication via the $i$-th embedding.

\subsection{Harder-Narasimhan filtration and degrees of eigenspace bundles}

We order the line bundles $\LBundle_i$ in such a way that $\LBundle_1$ is the one that is generated by $\omega_s$ (and is therefore maximal Higgs by \cite{moeller06}), and such that $\LBundle_1 \oplus \LBundle_2$ is the next step in the Harder-Narasimhan filtration of $\VBundle$ (see \cite[Prop. 4.3]{bainbridgehabeggermoeller}).

This ordering is the leading normalization used in the whole paper. The ordering of the Galois embeddings made above fits together the ordering of the corresponding eigenspace bundles  as one can see from the next remark and by looking at the differentials over the cusps given in Section~\ref{sec:specialpoints}.

\begin{rem} \label{rem:divisors-of-omegai}
	The ordering of the line bundles $\LBundle_i$ is also reflected in the divisors of its sections. If we let $\omega^{(i)}$ denote a local section of $\LBundle_i$ ($i=1,2,3$), then by \cite[Prop. 4.1]{bainbridgehabeggermoeller}
	\begin{align}
	\begin{split}\label{eqn:divisors-of-omegai}
	\div(\omega_s^{(1)}) &= 3P_s + Q_s \in \Pic(\ol\XFam_s)\\
	\div(\omega_s^{(2)}) &\geq P_s \in \Pic(\ol\XFam_s).
	\end{split}
	\end{align}
\end{rem}



Since $\ol{C}\cong \PP^1$, the line bundle $\LBundle_i$ is isomorphic to $\OX_{\PP^1}(k_i)$, where $k_i = \deg(\LBundle_i)$, and these degrees are readily computed.

\begin{lem}
	\label{lem:degrees}
	The degrees of the line bundles $\LBundle_i$ over $\ol{C}$ are given by
	\begin{align*}
	\deg(\LBundle_1) &= 4 & \deg(\LBundle_2) &= 2 & \deg(\LBundle_3) &= 1
	\end{align*}
\end{lem}
\begin{proof}
	First of all note that
	\[\deg(\Omega^1_{\ol{C}}(\log(S)) =-\chi(\tilde{C})=-(\chi(\ol{C})-\#\text{cusps})=8.\]
	Since we have a splitting of $\VBundle$ into line bundles, we can consider the numbers
	\[\lambda_i = \frac{2 \deg(\LBundle_i)}{\deg(\Omega^1_{\ol{C}}(\log(S))}=\frac14 \deg(\LBundle_i).\]
	Since $\LBundle_1$ is maximal Higgs, $\displaystyle \lambda_1 = 1$.

	The quotient of the degrees of the steps of the Harder-Narasimhan filtration and the degree of $\Omega^1_{\ol{C}}(\log(S)$ were computed  in  \cite[Table 1]{YZ13} for the $(3,1)$-Stratum 
	under the name $w_i$ of Weierstra\ss\ exponents by analyzing filtrations of the relative dualizing sheaf given by Weierstra\ss\ gap sequences.
	Using the same method, \cite[Prop. 4.3]{bainbridgehabeggermoeller} shows that the Harder-Narasimhan filtration of $\VBundle$ is given by the filtration of eigenspace bundles. 
	By our ordering choice, this filtration is 
	\[\LBundle_1\subset \LBundle_1 \oplus \LBundle_2\subset \VBundle.\] 
	
	Hence, since we established that the $\lambda_i$ are the same as $w_i$, we see from \cite[Table 1]{YZ13} that $\displaystyle \lambda_2 = \frac12,\ \lambda_3 = \frac14$ and we can conclude.
	
\end{proof}

\begin{rem}
	By Kontsevich's formula (see \eg \cite[Theorem 9.2]{bouwmoel}), the $\lambda_i$ are the Lyapunov exponents of the Kontsevich-Zorich cocycle on $C$.  
\end{rem}

\section{Special points of the Teichm\"uller curve}\label{sec:specialpoints}

In this section, we study the three special points of the Kenyon-Smillie Teichm\"uller curve. We compute how real multiplication acts in the orbifold point and we compute the equations of the orbifold point and the two cusps as plane quartics in the distinguished coordinate system given by sections in the eigenform line bundles.

\subsection{Orbifold point}

The orbifold point of the Teichm\"uller curve corresponds to the surface $(S,\omega)$, which admits an automorphism of order $9$. We want to compute the action of this automorphism on the space of holomorphic differentials.

\begin{prop}
\label{prop:orbifoldpt}
 The translation surface $(S,\omega)$ corresponding to the orbifold point on $C$ is the complete non-singular curve $\tilde S$ with singular model given by the affine equation
 \[y^9 = x^2(x-1)^3.\]
The curve  $\tilde S$ is a $9$-sheeted cyclic covering of $\PP^1$ totally branched over $0$ and $\infty$ and branched of order $3$ over $1$. The deck transformation is given by
 \[g: (x,y)\mapsto (x,\zeta_9\cdot y).\] 
 Up to scalar multiples, the eigendifferentials for real multiplication $\omega_0^{(i)}$ are given by
 \begin{align}\label{eqn:differentials-orbifoldpoint}
  \omega_0^{(1)} = \frac{y\, \dd x}{x(x-1)},\quad \omega_0^{(2)} = \frac{y^5\dd x}{x^2(x-1)^2},\quad \omega_0^{(3)} = \frac{y^7\dd x}{x^2(x-1)^3}.
 \end{align}
\end{prop}
\begin{proof}
 Up to isomorphism, there is only one curve in genus $3$ admitting a $\ZZ/9\ZZ$-action (cf. \cite{kuribayashi-komiya}), so $(S,\omega)$ and $\tilde S$ are isomorphic. 
To see that the given differentials form a basis of holomorphic differentials see \cite{bouwprank}.
The divisors of these differentials are given by
\[ \div(\omega_0^{(1)})=3p+q,\quad \div(\omega_0^{(2)})=p+\sum_{i=1}^3 l_i,\quad \div(\omega_0^{(3)})=4q\]
where $p$ is the preimage of $\infty\in \PP^1$ , $q$ is the preimage of $0\in \PP^1$ and $l_i$ are the three preimages of $1\in\PP^1$ under the cyclic covering map.

Note that the cyclic group generated by $g$ must be the same as the cyclic group generated by the automorphism of the flat surface $(S,\omega)$. Hence $\omega_0^{(1)}$ in \eqref{eqn:differentials-orbifoldpoint} is proportional to $\omega$, since it is the only eigendifferential of $g^*$  with the right divisor.  By the construction of real multiplication on  Teichm\"uller curves, we now know that its restriction to the fiber over $0$ is given by the action of the totally real subfield $\QQ(\zeta_9+\zeta_9^{-1})=\QQ(v)\leq\QQ(\zeta_9)$. Therefore, the $\omega_0^{(i)}$ given in \eqref{eqn:differentials-orbifoldpoint} are indeed eigendifferentials for the action of real multiplication, and from their divisors and expression~\eqref{eqn:divisors-of-omegai} we see that $\omega_0^{(i)}\in (\LBundle_i)_0$.
\end{proof}

\begin{cor}
\label{cor:matrixorb}
 The action of the order $9$ automorphism $g$ on $H^0(S,\Omega^1)$ with respect to the basis $\{\omega_0^{(i)}\}$ is given by the diagonal matrix
 \[A_g=\diag(\zeta,\zeta^5,\zeta^7) \in \PGL_3(\CC).\]
\end{cor}

Even though we will not need it in the next sections, it is nice to present the smooth quartic model of the curve $S$ given via the canonical embedding.

\begin{cor}
	Up to isomorphism, the  curve $S=S_0$ is given as the vanishing locus in $\PP^2$ of
\[	F_0(X,Y,Z) = X^4+XZ^3+3Y^3Z.\]
\end{cor}

\begin{proof}
	One checks that putting $\displaystyle X = -\zeta_3 \omega_0^{(1)}$, $Y = \sqrt[3]{\frac{\zeta_3}{3}}\omega_0^{(2)}$, $Z = \omega_0^{(3)}$, where $\zeta_3$ is a third root of unity, we obtain the above relation in $\Hcoh^0(S,\omega_{S}^4)$. The normalizing coefficients are put in order to have an  equation consistent with the choices made later on.
\end{proof}
\subsection{{Cusps}}
One of the two cusps of the Teichm\"uller curve is an irreducible stable curve, while the other one is reducible.
We choose $t$ such that the reducible stable curve lies over $t = \infty$ and the irreducible stable curve over $t = 1$.

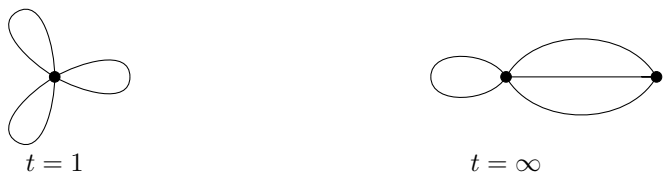
\begin{figure}[ht]
 \begin{center}
  \begin{tikzpicture}[scale=2]
   \begin{scope}
    \coordinate (v) at (0,0) {};
    \coordinate (v1) at (0:0.5);
    \coordinate (v2) at (120:0.5);
    \coordinate (v3) at (240:0.5);
    \draw[fill=black] (v) circle (1pt);
    \draw (v) to[out=-30, in=270] (v1);
    \draw (v) to[out=30, in=90] (v1);
    \draw (v) to[out=90, in=30] (v2);
    \draw (v) to[out=150, in=210] (v2);
    \draw (v) to[out=210, in=150] (v3);
    \draw (v) to[out=270, in=330] (v3);
    \node[below=6ex] (0,0) {$t=1$};
   \end{scope}
   \begin{scope}[shift={(3,0)}]
    \coordinate (v1) at (0,0); 
    \coordinate (v2) at (1,0);
    \node[below=6ex] (0.5,0) {$t=\infty$};
    \coordinate (a) at (180:0.5);
    \draw[fill=black] (v1) circle (1pt);
    \draw[fill=black] (v2) circle (1pt);
    \draw (v1) to[out=0, in=0] (0.9,0);
    \draw (v1) to[out=60, in=120] (v2);
    \draw (v1) to[out=-60, in=240] (v2);
    \draw (v1) to[out=120, in=90] (a);
    \draw (v1) to[out=240, in=270] (a);
   \end{scope}
  \end{tikzpicture}
 \end{center}
\caption{Dual graphs of the two cusps of the Teichm\"uller curve. The vertices represent the connected components and the edges correspond to the nodes of the stable curves associated with the cusps.}
\label{fig:dualgraphs}
\end{figure}

Recall that the topological type of a cusp is obtained by contracting the core curves of cylinders in the  cylinder decomposition of $(S,\omega)$ in the periodic direction corresponding to the cusp (see for example \cite[Prop. 5.9]{moelPCMI}).

\subsection{Reducible cusp}

The reducible cusp  corresponds to the vertical cylinder decomposition of $(S,\omega)$ shown in Figure~\ref{fig:vert-cylinderdecomp}. 

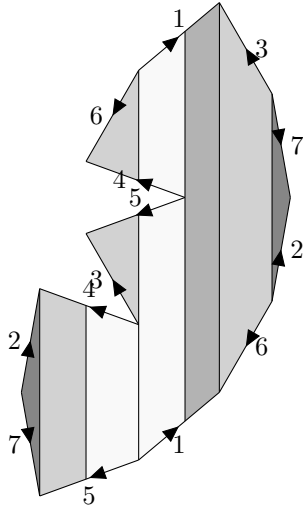
\begin{figure}[ht]
 \begin{center}
  \begin{tikzpicture}[scale=1.4, 
	ar/.style ={decoration={             
            markings, 
            mark=at position 0.5 with {\arrow{triangle 45}}
        }, postaction={decorate}}] 
\clip (-0.5,-0.2) rectangle (2.5,5.5);

\coordinate (p0) at (0,0);
\coordinate (p1) at (20:1);
\coordinate (p2) at ($(p1)+(40:1)$);
\coordinate (p3) at ($(p2)+(60:1)$);
\coordinate (p4) at ($(p3)+(80:1)$);
\coordinate (p5) at ($(p4)+(100:1)$);
\coordinate (p6) at ($(p5)+(120:1)$);
\coordinate (p7) at (100:1);
\coordinate (p8) at ($(p7)+(80:1)$);
\coordinate (p9) at ($(p8)+(340:1)$);
\coordinate (p10) at ($(p9)+(120:1)$);
\coordinate (p11) at ($(p10)+(20:1)$);
\coordinate (p12) at ($(p11)+(160:1)$);
\coordinate (p13) at ($(p12)+(60:1)$);
\coordinate (p14) at ($(p13)+(40:1)$);

\draw[ar, name path=side51] (p1)-- node[below] {$5$} (p0);
\draw[ar, name path=side11] (p1)-- node[below] {$1$} (p2);
\draw[ar, name path=side61] (p3)-- node[below,right] {$6$} (p2);
\draw[ar, name path=side21] (p3)-- node[right] {$2$} (p4);
\draw[ar, name path=side71] (p5)-- node[right] {$7$} (p4);
\draw[ar, name path=side31] (p5)-- node[right] {$3$} (p6);
\draw[ar, name path=side72] (p7)-- node[left] {$7$} (p0);
\draw[ar, name path=side22] (p7)-- node[left] {$2$} (p8);
\draw[ar, name path=side41] (p9)-- node[left, above] {$4$} (p8);
\draw[ar, name path=side32] (p9)-- node[left] {$3$} (p10);
\draw[ar, name path=side52] (p11)-- node[above] {$5$} (p10);
\draw[ar, name path=side42] (p11)-- node[left] {$4$} (p12);
\draw[ar, name path=side62] (p13)-- node[left] {$6$} (p12);
\draw[ar, name path=side12] (p13)-- node[left, above] {$1$} (p14);

\path[name path=p9vert] ($(p9)+(0,-10)$)--($(p9)+(0,10)$);\draw[name intersections={of=p9vert and side52,by=x1}] (p9)--(x1);
\draw[name intersections={of=p9vert and side42,by=x2}] (x2)--(p13);
\path[name path=p11vert] ($(p11)+(0,-10)$)--($(p11)+(0,10)$);
\draw[name intersections={of=p11vert and side12,by=x3}] (p11)--(x3);
\draw (p1)--(p9);
\draw[name intersections={of=p11vert and side11,by=x4}] (p11)--(x4);
\path[name path=p10vert] ($(p10)+(0,-10)$)--($(p10)+(0,10)$);
\path[name intersections={of=p10vert and side51,by=x5}];
\draw[name intersections={of=p10vert and side41,by=x6}] (x5)--(x6);
\draw[name path=p2vert] (p2)--(p6);
\draw[name path=p0vert] (p0)--(p8);
\draw[name path=p3vert] (p3)--(p5);
\begin{scope}[on background layer]
\fill[gray!5] (x5)--(p1)--(p9)--(x6)--cycle;
\fill[gray!5] (p11)--(x3)--(p13)--(x2)--cycle;
\fill[gray!5] (p1)--(x4)--(p11)--(x1)--cycle;

\fill[gray!35] (p0)--(x5)--(x6)--(p8)--cycle;
\fill[gray!35] (p9)--(x1)--(p10)--cycle;
\fill[gray!35] (p12)--(x2)--(p13)--cycle;
\fill[gray!35] (p2)--(p3)--(p5)--(p6)--cycle;

\fill[gray!60] (x4)--(p2)--(p6)--(x3)--cycle;

\fill[gray!95] (p0)--(p8)--(p7)--cycle;
\fill[gray!95] (p3)--(p5)--(p4)--cycle; 
\end{scope}
  \end{tikzpicture}
 \end{center}
\caption{Vertical cylinder decomposition of $S_{\infty}$ with cylinders $A$, $B$, $C$, $D$ (from light to dark).
\label{fig:vert-cylinderdecomp}
}
\end{figure}

\begin{prop}\label{prop:redcusp-differential}
 \begin{enumerate}[a)]
  \item  The stable curve $(S_\infty,\omega_\infty)$ consists of two $\PP^1$ 
 \[S_\infty = T\cup U\]
 that meet at $3$ points $B,C,D$.
 The only other singular point is a node $A$, where $T$ is glued to itself.
  \item The projective tuple of residues of the stable differential is given by
  \[(r_A:r_B:r_C:r_D) = (-v^2 - v + 3\ : \ 1\ : \ v^2 -3\ : \ -v^2 + 2).\]
  \item Up to scaling and isomorphism, the stable differential $\omega_\infty$ on the normalization is given by
  \begin{align}
   \label{eqn:omega1_red}
   \omega_\infty\big|_{T} &= \mu_\infty\left(\frac{r_A}{z-1} + \frac{-r_A}{z+1} + \frac{r_B}{z-B} + \frac{r_C}{z-C} + \frac{r_D}{z-D}\right)dz\\
   \omega_\infty\big|_{U} &= \mu_\infty\left(\frac{-r_B}{z} + \frac{-r_C}{z-1} + \frac{-r_D}{z+1}\right)dz
  \end{align}
  where
  \begin{align*}B &= \tfrac{1}{17}(-2v^2 + 6v + 5), & 
    C &= \tfrac{1}{17}(-6v^2 - 8v + 13), & 
    D &= \tfrac{1}{17}(8v^2 + 2v - 15),
  \end{align*}
  and $\mu_\infty = v^2 + 2v - 2$.
 \end{enumerate}
\end{prop}
\par
The factor $\mu_\infty$ in front is chosen in such a way that the equation for the stable curve in Corollary~\ref{cor:redcusp-equation} has rational coefficients.
\par
\begin{proof}
 The topological type of $S_\infty$ and the projective tuple of residues of the stable differential can be read  off directly from the flat picture. The cusp at $\infty$ corresponds to replacing the cylinders of the cylinder decomposition of $S$ in the vertical direction (see Figure~\ref{fig:vert-cylinderdecomp}) by infinitely long cylinders. The residues of the stable differential $\omega_\infty$ at the nodes are given by the circumferences of the cylinders up to simultaneous multiplication by a non-zero scalar.
 
 We identify the normalization of $T$ with $\PP^1$ by sending the two preimages of $A$ to $1$ and $-1$ and the triple zero of $\omega_\infty$ to $0$. One checks that there are precisely two differentials on $\PP^1$ 
 with simple poles at $1,-1$ and at three other points $B,C,D$ with the residues given above and a threefold zero at $P=0$. These two differentials corresponds to the triples
\[ B_1 = \tfrac{1}{17}(-2v^2 + 6v + 5),\quad C_1 = \tfrac{1}{17}(-6v^2 - 8v + 13),\quad  D_1 = \tfrac{1}{17}(8v^2 + 2v - 15)\]
 and
\[B_2 = \tfrac{1}{19}(6v^2 + 18v -21),\quad C_2 = \tfrac{1}{19}(-18v^2 - 12v + 27),\quad   D_2 = \tfrac{1}{19}(12v^2 - 6v - 33).\]
 After applying Galois conjugation  to the residues, exchanging $v=v^{(1)}$ with $v^{(2)}$, by Remark~\ref{rem:divisors-of-omegai} the Galois conjugate $\omega_{\infty}^{(2)}$  has to have  a zero at $z=0$. One checks that only the differential corresponding to the triple $(B_1,C_1,D_1)$ has this property.
 \par
 In the same way, we identify $U$ with $\PP^1$ by sending the three nodes $B$, $C$, $D$ to $0$, $1$ and $-1$ and we obtain $\omega_\infty\big|_{U}$.
\end{proof}

\begin{cor}\label{cor:redcusp-equation}
 Up to isomorphism, the stable curve $S_\infty$ is given as the vanishing locus in $\PP^2$ of
 \begin{align}\label{eqn:Fred}
 \begin{split}
 F_\infty(X,Y,Z) &= X^4 - 3X^3Y + 6X^3Z - 3X^2Y^2 - 6X^2YZ + 6X^2Z^2\\& + 4XY^3 -6XY^2Z -6XYZ^2 + XZ^3 + 3Y^4 + 3Y^3Z\\
 &= (X+Y+Z) \\
 &\cdot (X^3 -4X^2Y + 5X^2Z + XY^2 -7XYZ + XZ^2 + 3Y^3).
\end{split}
 \end{align}
\end{cor}
\begin{proof}
 One checks that putting $X = \omega_\infty^{(1)}$, $Y = \omega_\infty^{(2)}$, $Z = \omega_\infty^{(3)}$, where
 \begin{align*}
  \omega_\infty^{(i)}\big|_{T} &= \mu_\infty^{(i)}\left(\frac{r_A^{(i)}}{z-1} + \frac{-r_A^{(i)}}{z+1} + \frac{r_B^{(i)}}{z-B} + \frac{r_C^{(i)}}{z-C} + \frac{r_D^{(i)}}{z-D}\right)dz\\
  \omega_\infty^{(i)}\big|_{U} &= \mu_\infty^{(i)}\left(\frac{-r_B^{(i)}}{z} + \frac{-r_C^{(i)}}{z-1} + \frac{-r_D^{(i)}}{z+1}\right)dz\qquad (i=1,2,3)
 \end{align*}
 are the three Galois conjugates of $\omega_\infty$, we obtain the above relation in $\Hcoh^0(S_\infty,\omega_{S_\infty}^4)$.
\end{proof}

\subsection{Irreducible cusp}

The irreducible cusp and its stable differential has been described in \cite[Example 13.8]{bainmoel}. We recall their description. It is obtained from the horizontal cylinder decomposition of $S_0$ shown in Figure~\ref{fig:hor-cylinderdecomp} by replacing each half of a cylinder by a half-infinite strip of the same width.

\begin{figure}[ht]
 \begin{center}
  \begin{tikzpicture}[scale=1.4,
  	ar/.style ={decoration={             
            markings, 
            mark=at position 0.5 with {\arrow{triangle 45}}
        }, postaction={decorate}}] 
\clip (-0.5,-0.2) rectangle (2.5,5.5);
\coordinate (p0) at (0,0);
\coordinate (p1) at (20:1);
\coordinate (p2) at ($(p1)+(40:1)$);
\coordinate (p3) at ($(p2)+(60:1)$);
\coordinate (p4) at ($(p3)+(80:1)$);
\coordinate (p5) at ($(p4)+(100:1)$);
\coordinate (p6) at ($(p5)+(120:1)$);
\coordinate (p7) at (100:1);
\coordinate (p8) at ($(p7)+(80:1)$);
\coordinate (p9) at ($(p8)+(340:1)$);
\coordinate (p10) at ($(p9)+(120:1)$);
\coordinate (p11) at ($(p10)+(20:1)$);
\coordinate (p12) at ($(p11)+(160:1)$);
\coordinate (p13) at ($(p12)+(60:1)$);
\coordinate (p14) at ($(p13)+(40:1)$);

\draw[ar, name path=side51] (p1)-- node[below] {$5$} (p0);
\draw[ar, name path=side11] (p1)-- node[below] {$1$} (p2);
\draw[ar, name path=side61] (p3)-- node[below,right] {$6$} (p2);
\draw[ar, name path=side21] (p3)-- node[right] {$2$} (p4);
\draw[ar, name path=side71] (p5)-- node[right] {$7$} (p4);
\draw[ar, name path=side31] (p5)-- node[right] {$3$} (p6);
\draw[ar, name path=side72] (p7)-- node[left] {$7$} (p0);
\draw[ar, name path=side22] (p7)-- node[left] {$2$} (p8);
\draw[ar, name path=side41] (p9)-- node[above] {$4$} (p8);
\draw[ar, name path=side32] (p9)-- node[left] {$3$} (p10);
\draw[ar, name path=side52] (p11)-- node[left] {$5$} (p10);
\draw[ar, name path=side42] (p11)-- node[below] {$4$} (p12);
\draw[ar, name path=side62] (p13)-- node[left] {$6$} (p12);
\draw[ar, name path=side12] (p13)-- node[left, above] {$1$} (p14);

\path[name path=p1hor] ($(p1)+(-10,0)$)--($(p1)+(10,0)$);\draw[name intersections={of=p1hor and side72,by=x1}] (p1)--(x1);
\draw[name path=p2hor] (p2) -- (p7);
\path[name path=p9hor] ($(p9)+(-10,0)$)--($(p9)+(10,0)$);
\draw[name intersections={of=p9hor and side22,by=x2}] (p9)--(x2);
\draw[name intersections={of=p9hor and side61,by=x3}] (p9)--(x3);
\path[name path=p3hor] ($(p3)+(-10,0)$)--($(p3)+(10,0)$);
\draw[name intersections={of=p3hor and side32,by=x4}] (p3)--(x4);
\path[name path=p10hor] ($(p10)+(-10,0)$)--($(p10)+(10,0)$);
\draw[name intersections={of=p10hor and side21,by=x5}] (p10)--(x5);
\draw (p4)--(p11);
\path[name path=p12hor] ($(p12)+(-10,0)$)--($(p12)+(10,0)$);
\draw[name intersections={of=p12hor and side71,by=x6}] (p12)--(x6);
\path[name path=p5hor] ($(p5)+(-10,0)$)--($(p5)+(10,0)$);
\draw[name intersections={of=p5hor and side62,by=x7}] (p5)--(x7);
\path[name path=p13hor] ($(p13)+(-10,0)$)--($(p13)+(10,0)$);
\draw[name intersections={of=p13hor and side31,by=x8}] (p13)--(x8);

\begin{scope}[on background layer]
\fill[gray!5] (p1)--(p2)--(p7)--(x1)--cycle;
\fill[gray!5] (p12)--(x6)--(p5)--(x7)--cycle;
\fill[gray!5] (p7)--(p2)--(x3)--(x2)--cycle;
\fill[gray!5] (x4)--(p3)--(x5)--(p10)--cycle;
\fill[gray!5] (p13)--(x8)--(p6)--cycle;

\fill[gray!45] (p9)--(x3)--(p3)--(x4)--cycle;
\fill[gray!45] (x7)--(p5)--(x8)--(p13)--cycle;

\fill[gray!90] (p0)--(p1)--(x1)--cycle;
\fill[gray!90] (p10)--(x5)--(p4)--(p11)--cycle;
\fill[gray!90] (x2)--(p9)--(p8)--cycle;
\fill[gray!90] (p11)--(p12)--(x6)--(p4)--cycle;

\end{scope}
  \end{tikzpicture}
 \end{center}
 \caption{Horizontal cylinder decomposition of $S_0$ with cylinders $C_1$, $C_2$, $C_3$ (from light to dark).
  \label{fig:hor-cylinderdecomp}
}
\end{figure}
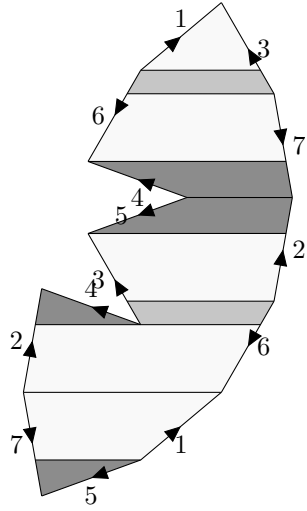

\begin{prop}\label{prop:irredcusp-differential}
 \begin{enumerate}[a)]
  \item The stable curve $S_1$ is isomorphic to a $\PP^1$ with $3$ pairs of points identified.
  \item The projective tuple of residues of the stable differential $\omega_1$ is given by
  \[(r_1 : r_2 : r_3) = (-v^2 -v \ : \ v + 1\ :\ -2v^2 -3v +2).\]
  \item Up to scaling and isomorphism, the stable differential $\omega_1$ on the normalization is given by
  \begin{align*}
   \omega_1 &= \mu_1 \cdot \sum_{i=1}^3 \biggl(\frac{r_i}{z-x_i} - \frac{r_i}{z-\zeta_3 x_i}\biggr),
  \end{align*}
  where $\zeta_3 = \exp(2\pi i/3)$, $\mu_1 = -v^2 + 2$ and
  \[x_1 = 1,\qquad x_2 = 2-v^2,\qquad x_3 = v^2 - 3.\]
 \end{enumerate}
\end{prop}
\par
Again, the factor $\mu_1$ is chosen in order that the coefficients of the equation below become rational.
\par
\begin{cor}\label{cor:irredcusp-equation}
 Up to isomorphism, the stable curve $S_1$ is given as the vanishing locus in $\PP^2$ of
 \begin{align}
    \begin{split}\label{eqn:Firred}
 F_1(X,Y,Z) &= 2X^4 - 3X^3Y + 6X^3Z - 3X^2Y^2 - 6X^2YZ + 6X^2Z^2\\& + 4XY^3 -6XY^2Z -6XYZ^2 + XZ^3 + 3Y^4 + 3Y^3Z
\end{split}
 \end{align}
\end{cor}
\begin{proof}
As in the proof of Corollary~\ref{cor:redcusp-equation}, one checks that putting $X = -\omega_1^{(1)}$, $Y = \omega_1^{(2)}$, $Z = \omega_1^{(3)}$, where  the $\omega_1^{(i)}$ are the three Galois conjugates of $\omega_1$, we obtain the quartic relation \eqref{eqn:Firred}. 
\end{proof}

Note that the choice of the coordinates in the previous corollary is made in such a way that the choice of coordinates of $\PP^2$ at $s=1$ is consistent with the choice of coordinates at $s=\infty$.  This will be clear after Proposition~\ref{prop:coeffshape}.

\section{Equation of the the universal family over the Teichm\"uller curve}\label{sec:equation}

In this section we will prove the main Theorem \ref{thm:mainthm}.

\subsection{The Teichm\"uller curve as a family of plane quartics}
We want to realize the family $\ol{\phi}:\ol{\XFam}\to \ol{C}$ as a family of plane quartics via the canonical embedding.
Since some of the fibers are stable curves, we need to employ the relative dualizing sheaf $\omega_{\ol{\XFam}/\ol{C}}$. 
It coincides with the relative canonical sheaf off the singular fibers. On a singular stable curve, its sections can be understood as meromorphic $1$-forms on the normalization, whose only poles are simple poles at the preimages of a node and such that the residues have opposite sign.

For every $s\in\PP^1$, let
\[\varphi_s: \ol{\XFam}_s \to \PP(\Hcoh^0(\ol{\XFam}_s,\omega_{\ol{\XFam}_s})^\dual),\qquad x\mapsto (\omega \mapsto \omega(x))\]
be the canonical map associated with $\ol{\XFam}_s$.
Note that if we fix a basis $\{\omega_s^{(i)}\}_{i=1,2,3}$ of $\Hcoh^0(\ol{\XFam}_s,\omega_{\ol{\XFam}_s})$, the map above is simply given by
\[\varphi_s:\ol{\XFam}_s \to \PP^2,\quad x\mapsto (\omega_s^{(1)}(x):\omega_s^{(2)}(x):\omega_s^{(3)}(x)).\]

\begin{prop}\label{prop:canonicalmap-embedding}
	The canonical map $\varphi_s$ is an embedding for every $s\in\PP^1$.
	As a consequence, each $\varphi_s(\ol{\XFam}_s)$ is a plane quartic.
\end{prop}

In order to prove the above proposition, we only have to check that there are no hyperelliptic curves in the family.

\begin{lem}\label{lem:nothyperell}
	None of the curves $\ol{\XFam}_s$ is hyperelliptic.
\end{lem}
\begin{proof}
 Since the Teichmüller curve is in the stratum $(3,1)$, the non-singular fibers cannot be hyperelliptc. Indeed if they were hyperelleptic, the flat differential would be an eigendifferential for the hyperelliptic involution and so the involution would fix or exchange its zeros. 
 Being of different order, the zeros cannot be exchanged, thus they would be Weierstrass points, which is impossible since they have odd order.
 
The two cusps cannot be hyperelliptic either because of the shape of their equations described in Corollaries \ref{cor:irredcusp-equation} and \ref{cor:redcusp-equation}. 
\end{proof}

\begin{proof}[Proof of Proposition~\ref{prop:canonicalmap-embedding}]
	For every $s$ outside the set of cusps, the canonical map is an embedding onto a non-singular plane quartic in $\PP^2$, since by Lemma~\ref{lem:nothyperell} every $\ol{\XFam}_s$ is a non-hyperelliptic curve. Since the singular fibers  $\ol{\XFam}_{\zeta_9^i}$ and $\ol{\XFam}_{\infty}$ are 3-connected (see Figure \ref{fig:dualgraphs}),  by \cite[Theorem 3.6]{cafrhu} or \cite[Theorem 1.2]{artamkin} their canonical map is still an embedding  and the image  is still a quartic.
\end{proof}



\subsection{Setup for the equation}

The canonical embedding of the fibers of the family $\ol{\phi}:\ol{\XFam}\to \ol{C}$ provides a rational map 
\[\varphi: \ol{\XFam} \to \PP^2\times \ol{C},\quad x_s\mapsto ((\omega^{(1)}_s(x_s):\omega^{(2)}_s(x_s):\omega^{(3)}_s(x_s)),s)\]
which is  an isomorphism onto its image. By Lemma~\ref{lem:degrees}, the global sections $\omega^{(i)}$ can be chosen such that
\begin{align} \label{eqn:choice-of-basis}
 \div(\omega^{(i)})=k_i\cdot \infty, \quad (i=1,2,3),
\end{align}
where $k_1=4$, $k_2=2$ and $k_3=1$. In the sequel, it will be convenient to use projective coordinates $(s_1:s_2)\in\PP^1$ for the base, so that $s\in\CC$ is identified with $(1:s)\in \PP^1$.

The image of the map $\varphi$ is the zero locus  of a degree $4$ primitive homogeneous polynomial  whose coefficients are homogeneous polynomials in $(s_1:s_2)\in \PP^1$. 
We denote this polynomial by
\[F = \sum_{i+j+k= 4} a_{i,j,k}(s_1,s_2) X^iY^jZ^k\in \CC[s_1,s_2][X,Y,Z]_4.\]
The aim  is to compute the coefficients $a_{i,j,k}(s_1,s_2)$.

\subsection{Degrees of the coefficients}

In this section, we exhibit the main ingredient, which is central in the proof of Theorem~\ref{thm:mainthm}. 
Using a structural result for the relative canonical ring of a family of quartics (cf. \cite[Proposition 7.9, (2)]{CatanesePignatelli2006}), we compute the degrees of the coefficients  $a_{i,j,k}(s_1,s_2)$.

\begin{prop}\label{prop:degrees-of-coeff}
 For every $(i,j,k)$ with $i+j+k = 4$, the coefficient $a_{i,j,k}$ of the primitive polynomial $F$ is a homogeneous polynomial in $(s_1:s_2)\in \PP^1$ of degree
 \[\deg(a_{i,j,k}(s_1,s_2)) = 4i + 2j + k - 7\]
 if $4i + 2j + k - 7\geq 0$ or the zero polynomial otherwise.
\begin{table}[ht]
\begin{center}
 \begin{tabular}{|c|c|c|c|c|c|c|c|c|c|c|c|c|c|c|c|}
  \hline
  $i$ &4 & 3 & 3 &2 & 2 & 2 & 1 & 1 & 1 &1 &0&0&0&0&0\\
  \hline
  $j$ &0 & 1 & 0 &2 & 1 & 0 & 3 & 2 & 1 &0 &4&3&2&1&0\\
  \hline
  $k$ &0 & 0 & 1 &0 & 1 & 2 & 0 & 1 & 2 &3 &0&1&2&3&4\\
   \hline
   \hline
    $4i+2j+k-7$& 9 & 7 & 6 & 5 & 4 & 3 & 3 & 2 & 1 & 0 & 1 & 0 & -1 & -2 &-3\\
    \hline
    
 \end{tabular}
\end{center}
 \caption{Degree of $a_{i,j,k}(s_1,s_2)$.\label{table:degree-a-ijk}} 
\end{table}

\end{prop}

An important element in the proof of this proposition is the fact that there is an identification between maps $\OX_{\PP^1}(l)\to \OX_{\PP^1}(m)$ and global sections of $\OX_{\PP^1}(m-l)$, which are identified with homogeneous polynomials  if we consider the standard trivialization of $\OX_{\PP^1}(m-l)$.
The choice of global sections \eqref{eqn:choice-of-basis} is made exactly in order to have this identification, when we consider maps of line bundles $\LBundle_i\to \LBundle_j$ and their canonical trivializations given by  taking $\omega^{(i)}$ as a basis over $\PP^1-\{\infty\}$ and $s_2^{k_i}\cdot \omega^{(i)}$ over $\PP^1-\{0\}$, where $k_i = \deg(\LBundle_i)$, $i=1,2,3$.

\begin{proof}
 As in \cite{CatanesePignatelli2006}, we consider the relative canonical algebra associated with $\ol{\phi}:\ol{\XFam}\to \ol{C}$ whose graded pieces are direct images of powers of the relative dualizing sheaf
 \[V_n =\ol{\phi}_*\omega^{\otimes n}_{\ol{\XFam}/\ol{C}}\]
  and the multiplication map 
 \[\sigma_n : \Sym^n(V_1) \to V_n\]
 given on a fiber over $s\in \ol{C}$ by
\[  (\sigma_n)_s\colon \Sym^n(\Hcoh^0(\ol{\XFam}_s, \omega_{\ol\XFam_s}))\to \Hcoh^0(\ol{\XFam}_s,\omega_{\ol\XFam_s}^{\tensor n}),\quad \omega_1\otimes\dots\otimes \omega_n\mapsto \omega_1\cdots \omega_n.\]
 
 Since the canonical embedding of $\ol{\XFam}_s$ is a plane quartic defined by the equation $F_s$, the standard exact sequence of sheaves induced by the embedding is given by 
\[0\to \mathcal{I}(F_s)(4)\to \OX_{\PP(\Hcoh^0(\ol{\XFam}_s,\omega_{\ol{\XFam}_s})^\dual)}(4)  \to \omega_{\ol\XFam_s}^{\tensor 4}\to 0.\]
This shows that  the kernel $\KBundle_4$ of $\sigma_4$ is precisely the ideal sheaf generated by $F$.

Now we can use  \cite[Proposition 7.9, (2)]{CatanesePignatelli2006}, which states that in the case of a family of genus $3$ non-hyperelliptic curves, the kernel $\KBundle_4$ of $\sigma_4$ is isomorphic to the determinant bundle of the push forward of the relative dualizing sheaf
\[ \KBundle_4\isom \det(V_1).\] 
 \par
 To extract the coefficients of $F$, we look at the projections onto the direct factors of $\Sym^4(V_1)$ given by the decomposition $V_1 = \bigoplus_{i=1}^3 \LBundle_i$ into eigenspace bundles for real multiplication.
 Let
 \[\delta(i,j,k) :\KBundle_4 \isom \det(V_1) \isom \LBundle_1\tensor \LBundle_2\tensor \LBundle_3  \to \Sym^4(V_1) \to \LBundle_1^{\tensor i}\tensor \LBundle_2^{\tensor j} \tensor \LBundle_3^{\tensor k}\]
 denote the composition of the above maps.
 By Lemma~\ref{lem:degrees},
 \[\LBundle_1\tensor \LBundle_2\tensor \LBundle_3 \isom \OX_{\PP^1}(7)\]
 and
 \[\LBundle_1^{\tensor i}\tensor \LBundle_2^{\tensor j} \tensor \LBundle_3^{\tensor k} \isom \OX_{\PP^1}(4i+2j+k)\]
 where the isomorphisms are determined by the choice of the basis \eqref{eqn:choice-of-basis}.
 Thanks to our choice of isomorphisms,  the map $\delta(i,j,k)$ is given by a section of $\OX_{\PP^1}(4i+2j+k-7)$. Therefore,
 \[\delta(i,j,k)(\omega^{(1)}\tensor \omega^{(2)} \tensor \omega^{(3)}) 
      = a_{i,j,k}(s) \cdot {\omega^{(1)}}^{\tensor i} \tensor {\omega^{(2)}}^{\tensor j} \tensor {\omega^{(3)}}^{\tensor k}\]
where $a_{i,j,k}$ is a homogeneous polynomial in $(s_1:s_2)\in \PP^1$ of degree $4i+2j+k-7$ or $0$. By our construction, they are the coefficient of $F$ that we were searching for.
\end{proof}

Note that the equality 
\[a_{0,0,4} = a_{0,1,3} = a_{0,2,2} = 0\]
can be easily deduced also from  the form of the divisors of $\omega_s^{(i)}$ given by Condition~\eqref{eqn:divisors-of-omegai} with a local computation around the zeroes of the differentials.

\subsection{Conditions from the orbifold point}

Since the family over $\ol{C}$ comes from a degree $9$ covering of the original Teichm\"uller curve, the coefficients of the polynomial describing this family must have an order $9$ symmetry, which we want to understand.
We use this symmetry in order show that most of the coefficients of the polynomials $a_{i,j,k}(s_1,s_2)$ vanish.

\begin{prop}
\label{prop:coeffshape}

The coefficients of the dehomogenized polynomial $a_{4,0,0}(1:s)$ are all zero with the exception of the top and the constant one, while for all the other non-trivial triples $(i,j,k)$, all the coefficients of  $a_{i,j,k}(1:s)$ are zero other than the top one, namely
\[a_{i,j,k}(1:s) = \begin{cases} \alpha_{4,0,0}\cdot s^9+\beta ,& (i,j,k)=(4,0,0)\\
\alpha_{i,j,k}\cdot s^{4i+2j+k-7},& \text{otherwise}
\end{cases}\]
for $\alpha_{i,j,k},\beta\in \CC$.
\end{prop}

\begin{rem} The lifting of the action of the Galois group of the $9$-sheeted covering $\tilde C\to C$ to an action on the family $\ol{\phi}:\ol{\XFam}\to \ol{C}$ can be seen as follows.

A generator of the Galois group has a lift to an element of the Veech group, corresponding to an elliptic automorphism of $\HH$ of order $9$. Since the Veech group embeds naturally into the mapping class group, we obtain in fact an automorphism of the Teichm\"uller space $\Teich_3$. We pass to an appropriate quotient $\Moduli'$ of $\Teich_3$ that is a fine moduli space finitely covering $\Moduli_3$ and such that we have a factorization $\tilde C\to \Moduli'$ and obtain an automorphism of $\Moduli'$ fixing the image of $\tilde C$. By the universal property of $\Moduli'$, being a fine moduli space, this automorphism lifts to an automorphism of the universal family over $\Moduli'$, hence its restriction to $\tilde C$ gives an automorphism $\ol{\XFam}\to\ol{\XFam}$.
\end{rem}

We fix the Galois automorphism 
\[h:\ol{C}\to \ol{C},\quad s\mapsto \zeta_9^2\cdot s\]
of the base.
The lifted automorphism $H:\ol{\XFam}\to\ol{\XFam}$ over $h$ induces an automorphism of $\VBundle$, which in the basis \eqref{eqn:choice-of-basis} is given as an element of $A_H \in \PGL_3(\CC(s))$. The proof of Proposition~\ref{prop:coeffshape} essentially boils down to using the identity
\[\lambda \cdot F_s(X,Y,Z) = F_{h(s)} \circ A_H(X,Y,Z)\]
with $\lambda\in \CC^\times$. 
We first need a lemma on the shape of the matrix $A_H$.

\begin{lem}
\begin{enumerate}[(a)]
 \item $H$ preserves the eigenspace bundles $\LBundle_i$ and thus acts on the basis \eqref{eqn:choice-of-basis} by a diagonal matrix.
 \item The matrix $A_H$ is constant and given by
\[A_H = \diag(\zeta_9: \zeta_9^5: \zeta_9^7) \in \PGL_3(\CC).\]
\end{enumerate}
\end{lem}
\begin{proof}
 That $H$ preserves the eigenspace bundles $\LBundle_i$ follows from the $\LBundle_i$ being irreducible and from Schur's Lemma.
The diagonal entries of $A_H$ must be constant in $s$ since they are maps between line bundles on $\PP^1$ of the same degree. Hence, we just need to find the action of $A_H$ on one fiber in order to determine it completely.

 Consider now the fiber over $0$, which corresponds to the curve with the order $9$ automorphism. The  lifting automorphism $H:\ol{\XFam}\to\ol{\XFam}$ must specialize to a primitive element of the automorphism group of $\ol\XFam_0$.
 Thus by Corollary~\ref{cor:matrixorb}, the matrix $A_H$ must specialize to a primitive element of the group generated by the matrix $A_g$.
 In order to determine this element, we compute how $A_H$ acts on ${\LBundle_1}$.

Since $\LBundle_1$ is maximal Higgs, it is a Theta characteristic, \ie 
\[\LBundle_1^2\cong \Omega^1_{\ol{C}}.\] 
Since the  action of $H$ on ${(\Omega^1_{\ol{C}})}_0$ is  given by multiplication by $\zeta_9^2$,  the action on  ${(\LBundle_1)}_0$ must be given by multiplication by $\zeta_9$.
This fixes the element in the group generated by the matrix $A_g$ of Corollary~\ref{cor:matrixorb}, and so $A_H=A_g=\diag(\zeta_9: \zeta_9^5: \zeta_9^7) \in \PGL_3(\CC)$.
\end{proof}

Now we can use the automorphism $H$ in order to find symmetries of the polynomial $F$.

\begin{proof}[Proof of Proposition~\ref{prop:coeffshape}]
The existence of $H$ implies that  the fiber $\ol{\XFam}_s$ is isomorphic to the fiber $\ol{\XFam}_{h(s)}$. However, since in our case we are considering  canonical embeddings, we know more, namely that the locus of the polynomial $F_{h(s)}$ must be the same as the one of $F_s$ up to a projective linear transformation of $\PP^2$.  Recall that with the fixed  isomorphism $ \PP(\Hcoh^0(X,\Omega^1)^\dual)\isom \PP^2$ given by the choice of basis \eqref{eqn:choice-of-basis},  the canonical embedding is given by the map 
\[\varphi_s:\ol{\XFam}_s\to \PP^2,\quad  \varphi_s(x)=(\omega_s^{(1)}(x):\omega_s^{(2)}(x):\omega_s^{(3)}(x))\]
and so the projective linear transformation of $\PP^2$ that we are looking for is  given by the matrix $A_H$.

Hence the condition imposed by the existence of $H$ is 
\begin{align}\label{eqn:orbcond}
\lambda \cdot F_s(X,Y,Z) = F_{h(s)} \circ A_H(X,Y,Z)
\end{align}
for some $\lambda\in\CC^\times$. 

In order to compute $\lambda$, we look at the condition for the constant coefficient $a_{0,3,1}$. This coefficient cannot be zero, since it is not zero for the fiber over $1$ or $\infty$ by Corollary~\ref{cor:redcusp-equation} or Corollary~\ref{cor:irredcusp-equation}.
Condition~\eqref{eqn:orbcond} yields
\[\lambda \cdot a_{0,3,1}=a_{0,3,1}\cdot \zeta_9^{3\cdot 5+7}=\zeta_9^{4}\cdot a_{0,3,1},\]
whence we have
\[\lambda=\zeta_9^{4}.\]
Therefore, the other coefficients must satisfy 
\[\zeta_9^4\cdot  a_{i,j,k}(1,s)=\zeta_9^{i+5\cdot j + 7\cdot k} a_{i,j,k}(1,\zeta_9^2 s).\]
Using this condition together with Proposition~\ref{prop:degrees-of-coeff}, it is immediate to conclude.
\end{proof}

\subsection{Computation of the family}
Now we have gathered all the ingredients in order to prove Theorem~\ref{thm:mainthm} and write down the algebraic equation of the universal family over the Kenyon-Smillie-Teichm\"uller curve.
\begin{proof}[Proof of Theorem \ref{thm:mainthm}]
By Proposition \ref{prop:coeffshape}, it is sufficient to know the values of the coefficients $a_{i,j,k}(1:s)$ in two special point in order to determine them completely.
Using the computation of $a_{i,j,k}(0:1)$ over the special point $\infty=(0:1)\in \ol{C}$ given by Corollary~\ref{cor:redcusp-equation}, we can compute all coefficients other that $a_{4,0,0}(1,s)$, of which we have determined just the constant term. 
Now we can use Corollary~\ref{cor:irredcusp-equation}, where we computed the coefficients $a_{i,j,k}(1:1)$. The coefficient of the monomial $X^4$ of \eqref{eqn:Firred} tells us the sum of the top and the constant coefficient of $a_{4,0,0}(1:s)$.
The explicit form of the universal family $\ol{\phi}:\ol{\XFam}\to \ol{C}$ is then given by
\begin{equation}\label{eq:spar}
\begin{split}
(s^9 + 1)X^4  - 3s^7  X^3 Y + 6s^6  X^3 Z - 3 s^5  X^2 Y^2 - 6 s^4  X^2 Y Z + s^3(6  X^2 Z^2 + 4 X Y^3)  \\
 \quad - 6 s^2  X Y^2 Z + s(-6X Y Z^2 + 3Y^4)  + X Z^3 + 3Y^3 Z=0.
\end{split}
\end{equation}

If we exclude $s=0$ and $s=\infty$, we can apply the projective linear transformation 
\[(X:Y:Z)\mapsto (X:s^2 Y:s^3 Z)\]
 of $\PP^2$ and see that the family $\ol{\phi}:\ol{\XFam}\to \ol{C}$ indeed descends to the family $\phi:\XFam \to C - \{0\}$ described by Equation~\eqref{eq:maineq}.

We now identify, with a slight abuse of notation, $P_t$ with $\varphi_t(P_t)$ and $Q_t$ with $\varphi_t(Q_t)$, where $\div(\omega_t^{(1)}) = 3P_t + Q_t$.
Then, by Condition~\eqref{eqn:divisors-of-omegai},
\[P_t = (0:0:1)\quad \text{and}\quad Q_t = (0:\ast:\ast).\]
Hence, the point $P_t$ is given by the intersection of the curve $\XFam_t$ with the line $X=0$, thus $P_t = (0:1:-1)$. 
\end{proof}

\subsection{The torsion map and the hyperflex}

The torsion map is the map (unique up to multiplication by an element of $\CC^\times$) exhibiting the torsion condition $n(P_t-Q_t)=0\in \Pic^0(\XFam_t)$.

In order to compute it, we need to find a map  $\XFam_t \to \PP^1$ totally ramified at $P_t$ and $Q_t$. In fact, \cite[Section 13]{bainmoel} suggests that the degree of this map is $n=3$, since this is the case for the torsion map over the irreducible cusp.

\begin{proof}[Proof of Proposition \ref{prop:torsion}]
From Equation \ref{eq:maineq}, it is easy to check that the point $Q_t$ is a hyperflex and the tangent at $P_t$ is a triple tangent passing through $Q_t$. It is trivial now that the projection from $Q_t$ onto a line is the desired torsion map.
The explicit expression of the torsion map can be given by projecting from $Q_t$ onto the line $\{Z=0\}\subset \PP^2$. 
\end{proof}
 
 It is interesting to notice how Proposition~\ref{prop:torsion} can be proven without using Equation \ref{eq:maineq}, only using that we are in the $(3,1)$-stratum and that the points of the Teichmüller curve are canonically embedded as quartics in $\PP^2$.
 
 \begin{proof}[Alternative proof of Proposition \ref{prop:torsion}]
  The torsion map is a degree $3$ map  $\XFam_t \to \PP^1$ totally ramified at $P_t$ and $Q_t$. It is not difficult to prove that every degree $3$ map from a smooth quartic of $\PP^2$ to $\PP^1$ is a central projection from a point $x$ on the quartic. Since $\div(\omega_t)=3P_t+Q_t$ and the curves $S_t$ are canonically embedded, the line $\overline{P_tQ_t}$ is a triple tangent at $P_t$. This triple tangent is also the projection line $\overline{xP_t}$ since the torsion map is  totally ramified at $P_t$. Hence $x=Q_t$, namely the torsion map is the central projection from $Q_t$. Moreover, since $Q_t$ is a totally ramified point of this map, it must be a hyperflex of the quartic curve $\XFam_t$.
 \end{proof}

\subsection{Torsion map and real multiplication}
\label{sec:torsion}

Now we want to briefly explain the relation between the torsion map and real multiplication in the case of Veech-Ward-Bouw-M\"oller Teichm\"uller curves and compare this situation with ours. 

Let us recall the Bouw-M\"oller construction  of the universal family $\HFam\to \PP^1$ over the completion of the Teichm\"uller curve uniformized by a triangle group. 
For simplicity, we restrict to the case $\Delta(n,\infty,\infty)$ discussed in \cite[Section 5]{bouwmoel}. 
They construct a particular family of cyclic coverings $\YFam_t\to\PP^1$, parametrized by $t\in\PP^1 - \{0,1,\infty\}$, which is branched over $4$ points and admits an involution $\sigma$. This family  descends to the universal family $\HFam:=\YFam/\langle \sigma\rangle\to \PP^1$ over the Teichm\"uller curve. One can  construct the following  commutative diagram
     \begin{center}
 $\xymatrix{
\YFam_t \ar[d]_{\ZZ/ 2\ZZ} \ar[r]^{\ZZ/n\ZZ} &  \PP^1 \ar[d]\\
\HFam_t \ar[r]^{\Tor}         & \PP^1}$
\end{center}
where $\Tor$ is the torsion map.

Note that, even though we constructed the above diagram starting from the upper $n$-cyclic covering, one can build it also only from the torsion map, since this is indeed its Galois normalization diagram.
The relation between real multiplication and torsion map is now easy to explain. Let $\tau:\YFam_t\to \YFam_t$ be the Galois automorphism of $\YFam_t\to \PP^1$. It induces complex multiplication by $\QQ(\zeta_n)$ on the Jacobian of $\YFam_t$.
One can check that real multiplication is given by the push-forward to $\HFam_t$ of the correspondence induced by the graph of the endomorphism  $ \tau + \tau^*\in \End_{\QQ}\Jac(\YFam_t)$, which generates the totally real subfield of $\QQ(\zeta_n)$. 

Since by Proposition~\ref{prop:torsion} we know the explicit form of the torsion map in the Kenyon-Smillie-Teichm\"uller curve case, we can compute its Galois normalization and check if real multiplication comes from a diagram analogous to the one constructed in the Veech-Ward-Bouw-M\"oller case.

Note that one can easily compute the ramification points of the torsion map. There are the two expected triple ramification points $\Tor(Q_t)$ and  $\Tor(P_t)$ and other $6$ double ramification points, which we denote by $x_i(t)\in \PP^1$ for $i=1,\dots,6$.
Let $\mathcal{Z}_t$ be the genus $2$ curve defined by the affine equation
\[ \mathcal{Z}_t: \qquad v^2=\prod_{i=1}^6 (x-x_i(t)).\]

By general considerations on how the Galois normalization has to behave on the ramification points of the map $\Tor$, we find that the normalization diagram is given by the fibered product diagram
\begin{center}
 $\xymatrix{
 \YFam_t \ar[d] \ar[r]^{\ZZ/3\ZZ} &  \mathcal{Z}_t \ar[d]^{p}\\
\XFam_t \ar[r]^{\Tor}          & \PP^1}$
\end{center}
where $\YFam_t:=\XFam_t \times_{\PP^1} \mathcal{Z}_t$ is the fibered product.
The map 
\[p:\mathcal{Z}_t\to \PP^1,\quad (x,v)\mapsto v\]
  is the quotient by the hyperelliptic involution, thus it is ramified over the points $x_i$.

Now we want to show that real multiplication cannot be constructed from the Galois normalization diagram of the torsion map as in the Veech-Ward-Bouw-M\"oller case.
\begin{prop}
 Real multiplication is not given by a correspondence induced by the graph of an automorphism of $\YFam_t$ that descends to the quotient $\mathcal{Z}_t$.
\end{prop}
\begin{proof}
Recall that  the trace field of the Kenyon-Smillie-Teichmüller curve is cubic. Hence if real multiplication is given by a totally real subfield of the complex field induced by an automorphism of $\YFam_t$, this automorphism has to be of order $9$.  
Since the Galois automorphism of $\YFam_t\to \mathcal{Z}_t$ is of order $3$, this cannot work. We can already see that this is not as in the Bouw-M\"oller case.

Consider now an order $9$ covering automorphism of $\YFam_t$ descending to the quotient $\mathcal{Z}_t$.
Since the order of the descending automorphism on $\mathcal{Z}_t$ has to divide $9$, then it has to be of order $3$ because there are no order $9$ automorphisms on a genus $2$ curve.
We show that there is no such automorphism of $\mathcal{Z}_t$.

Recall that $p:\mathcal{Z}_t\to \PP^1$ is the hyperelliptic family. Hence if there were an order $3$ automorphism of $\mathcal{Z}_t$, it  would commute with the hyperelliptic involution and thus would descend to multiplication by a $3$rd root of unity on $\PP^1$ in the appropriate coordinates. Since the $6$ ramification points $x_i(t)$ must be preserved by the automorphism, they have to lie on two circles (any change of coordinates of $\PP^1$ preserves circles). Using Proposition~\ref{prop:torsion}, one can however compute explicitly $x_i(t)$ and check that this is not the case.
\end{proof}

\section{Picard-Fuchs equations}\label{sec:picard-fuchs}
In this section, we forget how we found Equation~\eqref{eq:maineq} 
and  we prove independently that this equation defines  a Teichm\"uller curve. 
We will do this by showing that the absolute cohomology bundle 
splits as a direct sum of three rank two subbundles and that one of them is maximal Higgs. 
The main tool  is the computation  of the Picard-Fuchs equation associated 
to a local section of the $(1,0)$-part of the absolute cohomology bundle of the family via the Griffiths-Dwork method.

\subsection{Griffiths-Dwork method}
We   quickly recall the Griffiths-Dwork algorithm for the  computation of the Picard-Fuchs equation of a family of projective hypersurfaces. This method can more generally be used in the case of projective toric varieties. See \cite{CoxKatz} for more details.

Let $V\subset \PP^n$ be a hypersurface of degree $d$ defined by a homogeneous equation $f=0$. We want to identify elements of $\Hcoh^{n-1}(V)$ with elements in $\Hcoh^n(\PP^n-V)$ via the residue map. Any element of $\Hcoh^n(\PP^n-V)$ can be represented by a form
\[\frac{P\Omega_0}{f^k},\quad \deg(P)=kd-(n+1)\]
where $\Omega_0$ is a section of the sheaf $\Omega^n_{\PP^n}(n+1)$, which in fact is trivial, and $P$ is a homogeneous polynomial. 
The residue map 
\[ \Res:\Hcoh^n(\PP^n-V)\to \Hcoh^{n-1}(V)\]
is defined by the property that
\[ \int_{\gamma} \Res\left( \frac{P\Omega_0}{f^k}\right)=\int_{T(\gamma)} \frac{P\Omega_0}{f^k}\]
for any $(n-1)$-cycle $\gamma$ in $V$ and its tubular neighborhood $T(\gamma)$.

Let $J(f):=\langle\partial f / \partial x_0,\dots,\partial f / \partial x_n\rangle$ be the Jacobian ideal of $V$.
The key ingredient of the Griffiths-Dwork method is the isomorphism
\[\left( \frac{\CC[x_0,\dots,x_n]}{J(f)}\right)_{kd-(n+1)} \cong \Prim\Hcoh^{n-k,k-1}(V) \quad\text{ for } k=1,\dots,n \]
where the subscript denotes the $kd-(n+1)$-graded piece and $\Prim\Hcoh^{n-k,k-1}(V)$ is the primitive part of $\Hcoh^{n-k,k-1}(V)$.\\
In other words, if a form $\frac{P\Omega_0}{f^k}$ has a high order pole, namely if $k>n$,
 we can find a representative of its cohomology class with a pole of order less than $n$.
The key equality that allows to compute representatives with lower order poles is given by
\begin{equation}\label{eq:grif-dwork}
\left( \sum_i G_i \frac{\partial f}{\partial x_j} \right)\frac{\Omega_0}{f^k}=\frac{1}{k-1}\left( \sum_i\frac{\partial G_i}{\partial x_j} \right)\frac{\Omega_0}{f^{k-1}} \in \Hcoh^n(\PP^n-V).
\end{equation}

Now consider a family of hypersurfaces $\{V_s\}$ defined by a varying polynomial $f_s$, and a form  represented  by  $\omega_s= \frac{P_s\Omega_0}{f_s^k}$. 
The action of the Gauss-Manin connection is given by
\[\nabla (\frac{\partial}{\partial s })(\omega_s)=\frac{(-kP_sf_s'+f_sP_s')\Omega_0}{f_s^{k+1}}.\]
If we iterate the derivation many times, we  find a form with poles of order greater than $n$. We can then find a representative of its cohomology class with a lower order pole by using Equation~\ref{eq:grif-dwork}.

\subsection{Picard-Fuchs equations of the family}

The Picard-Fuchs equation of a family of curves with respect to a local section of the $(1,0)$-part of its cohomology bundle is the differential equation satisfied by the periods of the chosen local section. The set of solutions of the Picard-Fuchs equation forms a local system isomorphic to the dual of the local system associated with the irreducible part of absolute cohomology bundle containing the chosen local section. Therefore, Picard-Fuchs equations characterize the local system underlying the absolute cohomology bundle in a unique way. 
A nice exposition of Picard-Fuchs equations can be found in \cite[Section 3]{bouwmoel}.
We will now prove Proposition~\ref{prop:P-F} using the Griffiths-Dwork method.

\begin{proof}[Proof of  Proposition \ref{prop:P-F}]
Let  $\ol{\phi}:\ol\XFam\to \ol{C}$ be the family of curves described by Equation~\eqref{eq:spar}. We use this family, since it has unipotent monodromy around the $10$ cusps. This will be useful, when we  check that the absolute cohomology bundle has a maximal Higgs rank $2$ subbundle.

Let $F_s$ be the polynomial describing $\ol\XFam_s\subset \PP^2$. Let $D:=\nabla (\frac{\partial}{\partial s })$, where $\nabla$ is the Gauss-Manin connection. 

  Let $\omega$ be a local section of the $(1,0)$-part of the cohomology bundle at $0\in \PP^1$. Since the cohomology bundle is of rank~$6$, the local section $\omega$ must a priori satisfy 
\[ a_0\omega+a_1 D(\omega)+\dots+ D^6(\omega)=0\]
where $a_i$, $i=0,\dots,5$, is a rational function in $s$ with  poles of order at most $6-i$.
The associated differential equation  
\[ a_0 y+a_1 \frac{\partial y}{\partial s} +\dots+ \frac{\partial^6 y}{{\partial s}^6}=0\]
is satisfied by the periods  $s\mapsto \int_{\gamma} \omega(s)$ for any locally constant $1$-cycle $\gamma$ of $\ol\XFam_s$.

We now want to use the identification of $\Hcoh^{1}(\ol\XFam_s)$ with $\Hcoh^2(\PP^2-\ol\XFam_s)$ via the residue map, and in particular the isomorphisms
\[\left( \frac{\CC[X,Y,Z]}{J(F_s)}\right)_{4k-3} \cong \Hcoh^{2-k,k-1}(\ol\XFam_s),\quad k=1,2.\]
Note that  in this case the primitive part is the full cohomology space.

Using the above identifications, one can check that  a basis for the cohomology groups is given by 
\[ \Hcoh^{1,0}(\ol\XFam_s)=\langle X\frac{\Omega_0}{F_s}, Y\frac{\Omega_0}{F_s},Z\frac{\Omega_0}{F_s} \rangle,\quad \Hcoh^{0,1}(\ol\XFam_s)=\langle X^5\frac{\Omega_0}{F_s^2}, Y^5\frac{\Omega_0}{F_s^2},Z^5\frac{\Omega_0}{F_s^2} \rangle\]
where the coefficient of $\frac{\Omega_0}{F_s^k}$ are considered modulo the Jacobian ideal.

We fix the local section $\omega(s)=X\frac{\Omega_0}{F_s}$ and we want to find the differential equation satisfied by the periods of $\omega$.
A priori, we are searching for an order $6$ differential equation, but we can check that indeed $\omega(s)$ satisfies an order $2$ differential equation. 
We want to find rational functions $a_0$ and $a_1$ such that
\[a_0(s)\int X\frac{\Omega_0}{F_s}+a_1(s)\int \frac{\partial}{\partial s}\left(X\frac{\Omega_0}{F_s}\right)+\int \frac{\partial^2}{{\partial s}^2}\left(X\frac{\Omega_0}{F_s}\right)=0.\]
We can get rid of the integral sign and rewrite the above equation in De Rham-cohomology as 
\[a_0(s) X\frac{\Omega_0}{F_s} -(a_1(s) F'(s)+F''(s)) X\frac{\Omega_0}{F_s^2} +2 {F'(s)}^2  X \frac{\Omega_0}{F_s^3}=0\in \Hcoh^1(\XFam_s).\]
Now we begin to apply Griffiths-Dwork method. 

We can check that $2 {F'(s)}^2  X$ is in the Jacobian ideal $J(F_s)$. Hence we can apply Equality~\eqref{eq:grif-dwork} and compute

 \[2 {F'(s)}^2  X \frac{\Omega_0}{F_s^3}=\left(G_1 \frac{\partial F_s}{\partial X}+G_1 \frac{\partial F_s}{\partial Y}+G_1 \frac{\partial F_s}{\partial Z}\right)\frac{\Omega_0}{F_s^3}=\frac{1}{2}G\frac{\Omega_0}{F_s^2}\in \Hcoh^1(\XFam_s)\] 
for $G=\left (\frac{\partial G_1}{\partial X}+\frac{\partial G_2}{\partial Y}+\frac{\partial G_3}{\partial Z}\right)$, where the $G_i(s)$ are some homogeneous polynomials in function of the rational parameter $s$.

We can now compute that
\[-XF'_s=\frac{9}{s}X^5 ,\quad \frac{1}{2}G-F''_sX= -\frac{81s^7}{s^9-1}X^5 \quad \text{in }  \frac{\CC[X,Y,Z]}{J(F_s)}.\]
Hence, we must have  that 
\[a_1(s)=\frac{9s^8}{s^9 - 1}.\]

By our choice of $a_1(s)$, we have
\[-(a_1(t) F'(s)+F''(s)) X\frac{\Omega_0}{F_s^2} +2 {F'(s)}^2  X \frac{\Omega_0}{F_s^3}=H\frac{\Omega_0}{F_s^2}\in \Hcoh^1(\XFam_s)\]
 for some $H=\left (H_1\frac{\partial F_s}{\partial X}+H_2\frac{\partial F_s}{\partial Y}+H_3\frac{\partial F_s}{\partial Z}\right)\in J(F_s)$.

Therefore, we can apply Equation~\ref{eq:grif-dwork} again finding that 
\[H\frac{\Omega_0}{F_s^2}=\left (\frac{\partial H_1}{\partial X}+\frac{\partial H_2}{\partial Y}+\frac{\partial H_3}{\partial Z}\right)\frac{\Omega_0}{F_s}=-\frac{16s^7}{s^9-1}X \frac{\Omega_0}{F_s}\in \Hcoh^1(\XFam_s).\]

Hence, by setting
\[a_0(s)=\frac{16s^7}{s^9-1}\]
we are done.

Since we have found a differential equation of order $2$, we have proved that the absolute cohomology bundle has a subbundle of rank $2$, whose dual is isomorphic to the space of solution of the differential equation
\begin{equation}\label{eq:diffeq}
 \frac{16s^7}{s^9-1} y+ \frac{9s^8}{s^9 - 1}y'+y''=0
\end{equation}
which has indeed regular singularities at the $9$-th roots of unity and at infinity.
It is easy to show that this differential equation is the pull-back of a unique hypergeometric differential equation
\[L_1(y)=\frac{16}{81t(t-1)}y+\frac{17t - 8}{9t(t - 1)}y'+y''=0\]

under the map $s\mapsto t=s^9$, which is the one in the statement of Proposition~\ref{prop:P-F}

Using the same algorithm, one can compute the differential equations satisfied by $\omega(s)^{(2)}=Y\frac{\Omega_0}{F_s}$ and $\omega(s)^{(3)}=Z\frac{\Omega_0}{F_s}$ and check that they are respectively the pull-back under the map $s\mapsto t=s^9$ of the unique differential equations 
\[ L_2(y)=\frac{4}{81t(t-1)}y+\frac{13t - 4}{9t(t-1)}y'+y''=0\]
\[ L_3(y)=\frac{1}{81t(t-1)}y+\frac{11t - 2}{9t(t - 1)}y'+y''=0.\]
 Hence we have proved that the cohomology bundle splits as a sum of three  rank $2$ subbundles.
\end{proof}

Now, we want to prove that the local system defined by the differential equation $L_1(y)=0$, or equivalently by the differential equation~\eqref{eq:diffeq}, is maximal Higgs. 

\begin{proof}[Proof of Corollary \ref{cor:maxhiggs}]

We can use \cite[Proposition 3.2]{bouwmoel} which relates the order of vanishing of the Kodaira-Spencer map to the local exponent of the associated differential equation in the case of unipotent monodromy.
One can easily compute the local exponent of the differential Equation~\eqref{eq:diffeq} and see that the Riemann scheme is given by
\begin{center}
\begin{tabular}{ c | c  |c|c|c|c|c|c|c|c  }
 $\zeta_9^1$& $\zeta_9^2$ &$\zeta_9^3$& $\zeta_9^4$&$\zeta_9^5$&$\zeta_9^6$&$\zeta_9^7$&$\zeta_9^8$&$\zeta_9^9$& $\infty$ \\
  \hline
  0 &0 &0 &0 &0 &0 &0 &0 &0 & 4 \\
  0 &0 &0 &0 &0 &0 &0 &0 &0 & 4 \\
\end{tabular}
\end{center}
Since Equation~\eqref{eq:diffeq} corresponds to  a family with unipotent monodromies, by \cite[Proposition 3.2]{bouwmoel},  the order of vanishing of the Kodaira-Spencer map at   these cusps is $0$. Hence  by definition, the associated local system is maximal Higgs and it has to be irreducible.  
\end{proof}

By  \cite[Theorem 5.3]{moeller06}, if the absolute cohomology bundle associated with a family of curves over $C$ has a maximal Higgs rank two subbundle, then the family is a finite unramified covering of a Teichm\"uller curve. By Proposition~\ref{prop:P-F}, this is the case for the family of curves described by Equation~\eqref{eq:maineq}.

\bibliographystyle{amsalpha}
\bibliography{teich31}
\end{document}